\documentclass[12pt, reqno]{amsart}
\usepackage{amsmath, amsthm, amscd, amsfonts, amssymb, graphicx}
\usepackage[bookmarksnumbered, colorlinks, plainpages]{hyperref}
\input{mathrsfs.sty}

\newtheorem{theorem}{Theorem}[section]
\newtheorem{lemma}[theorem]{Lemma}

\newtheorem{corollary}[theorem]{Corollary}
\theoremstyle{definition}

\theoremstyle{remark}

\numberwithin{equation}{section}

\begin{document}

\title{Non-commutative Callebaut inequality}
\author[M.S. Moslehian, J.S. Matharu, J.S. Aujla]{Mohammad Sal Moslehian$^1$, Jagjit Singh Matharu$^2$ and Jaspal Singh Aujla$^3$}

\address{$^1$ Department of Pure Mathematics, Center of Excellence in
Analysis on Algebraic Structures (CEAAS), Ferdowsi University of
Mashhad, P.O. Box 1159, Mashhad 91775, Iran.}
\email{moslehian@ferdowsi.um.ac.ir and moslehian@member.ams.org}

\address{$^2$ Department of Mathematics, Bebe Nanaki University College,
Mithra, Kapurthla, Punjab, India.}
\email{matharujs@yahoo.com}

\address{$^3$ Department of Mathematics, National Institute of Technology,
Jalandhar 144011, Punjab, India.}
\email{aujlajs@nitj.ac.in}



\begin{abstract}
We present an operator version of the Callebaut inequality involving
the interpolation paths and apply it to the weighted operator
geometric means. We also establish a matrix version of the Callebaut
inequality and as a consequence obtain an inequality including the
Hadamard product of matrices.
\end{abstract}

\maketitle
\section{Introduction and Preliminaries}

Let $\mathbb{B}(\mathscr{H})$ denote the algebra of all bounded linear operators acting on a complex
Hilbert space $(\mathscr{H},\langle \cdot,\cdot\rangle)$ and let $I$ be the identity operator. In the case when $\dim \mathscr{H} = n$, we identify $\mathbb{B}(\mathscr{H})$ with the full matrix algebra $\mathcal{M}_n(\mathbb{C})$ of all $n\times n$
matrices with entries in the complex field and denote its identity by $I_n$. The cone of positive operators is denoted by ${\mathbb B}({\mathscr H})_+$. We also denote the set of all invertible positive operators (positive definite matrices, resp.) by $\mathcal{P}$ ($\mathcal{P}_n$, resp.).

The axiomatic theory for operator means for positive operators
acting on Hilbert space operators was established by Kubo and Ando
\cite{KA}. A binary operation $\sigma: {\mathbb B}({\mathscr H})_+
\times {\mathbb B}({\mathscr H})_+ \to {\mathbb B}({\mathscr H})_+$
is called an \emph{operator mean} provided that
\begin{enumerate}
\item[(i)] $I\sigma I=I$;
\item [(ii)] $C^*(A\sigma B)C\leq(C^*AC)\sigma(C^*BC)$;
\item [(iii)] $A_n\downarrow A$ and $B_n\downarrow B$ imply $(A_n\sigma B_n)\downarrow A\sigma B$, where $A_n\downarrow A$ means that $A_1\geq A_2\geq \cdots$ and $A_n\to A$ as $n\to \infty$ in the strong operator topology;
\item [(iv)]  \begin{eqnarray}
\label{sub} A \leq B\,\, \&\,\, C \leq D \Longrightarrow A\sigma C \leq B\sigma D\,.
\end{eqnarray}
\end{enumerate}
It follows from the general theory of means \cite[Theorem 5.7]{seo})
that
\begin{eqnarray}
\label{add}\sum_{j=1}^{m}(A_{j}\sigma B_{j}) \leq (\sum_{j=1}^{m}A_{j})\sigma (\sum_{j=1}^{m} B_{j}).
\end{eqnarray}
There exists an affine order isomorphism between the class of operator means and the class of positive monotone operator functions $f$ defined on $[0,\infty)$ with $f(1)=1$ via $f(t)I=I\sigma (tI)\,\,(t\geq 0)$. In addition, $A \sigma B = A^{1/2}f(A^{-1/2}BA^{-1/2})A^{1/2}$. The operator monotone function $f$ is called the \emph{representing function} of $\sigma$. Using a limit argument by $A_\varepsilon=A+\varepsilon I$, one can extend the definition of $A\sigma B$ to positive operators (positive semidefinite matrices) as well. The dual $\sigma^\bot$ of an operator mean $\sigma$ with the representing function $f$ is the operator mean with representing function $t/f(t)$. So that $A \sigma^\bot B = A^{1/2}(A^{-1/2}BA^{-1/2}) f(A^{-1/2}BA^{-1/2})^{-1} A^{1/2}$. The operator means corresponding to the positive operator monotone functions $f_\sharp(t)=t^{1/2}, f_{\sharp_p}(t)=t^p\,\,(0 \leq p \leq 1)$ are the operator geometric mean $A\sharp B=A^{\frac{1}{2}}\left(A^{\frac{-1}{2}}BA^{\frac{-1}{2}}\right)^{\frac{1}{2}}A^{\frac{1}{2}}$ and the operator weighted geometric mean $A\sharp_pB=A^{\frac{1}{2}}\left(A^{\frac{-1}{2}}BA^{\frac{-1}{2}}\right)^{p}A^{\frac{1}{2}}$.

For an operator mean $\sigma$ satisfying $A\sigma B=B\sigma A$ a parametrized operator mean $\sigma_t$ is called an \emph{interpolational path} for $\sigma$ (see \cite[Section 5.3]{seo} and \cite{F-K}) if it satisfies
\begin{enumerate}
\item[(1)] $A\sigma_0 B=A$, $A\sigma_{1/2}B=A\sigma B$ and $A\sigma_1 B=B$;
\item[(2)] $(A\sigma_p B)\sigma(A\sigma_q B)=A\sigma_{(p+q)/2}B$, for all $p, q \in [0,1]$;
\item[(3)] the map $t \mapsto A\sigma_t B$ is norm continuous for each $A, B$.
\end{enumerate}
It is easy to see that the set of all $r\in [0,1]$ satisfying
\begin{eqnarray}
\label{sigma} (A\sigma_p B)\sigma_r(A\sigma_q B)=A\sigma_{rp+(1-r)q}B\,.
\end{eqnarray}
for all $p, q$ is a convex subset of $[0,1]$ including $0$ and $1$. Therefore \eqref{sigma} is valid for all $p, q, r \in [0,1]$, cf. \cite[Lemma 1]{F-K}.

The power means $Am_r B=A^{1/2}\left(\left(1+(A^{-1/2}BA^{-1/2})^r\right)/2\right)^{1/r}A^{1/2}\,\,(r\in [-1,-1])$ are typical examples of interpolational means, whose interpolational paths, by \cite[Theorem 5.24]{seo}, are
\begin{eqnarray}\label{path}
Am_{r,t}B=A^{\frac{1}{2}}\left(1-t+t\left(A^{\frac{-1}{2}}
BA^{\frac{-1}{2}} \right)^r \right)^{\frac{1}{r}}
A^{\frac{1}{2}}\quad (t \in [0,1])\,.
\end{eqnarray}
One of the fundamental inequalities in mathematics is the Cauchy--Schwarz inequality. There are many generalizations and applications of this inequality; see the monograph \cite{DRA}. There are some Cauchy--Schwarz type inequalities for Hilbert space operators and matrices involving unitarily invariant norms given by Joci\'c \cite{JOC} and Kittaneh \cite{KIT}. Morover, Niculescu \cite{NIC}, Joi\c{t}a \cite{JOI}, Ili\v sevi\'c--Varo\v{s}anec \cite{I-V}, Moslehian--Persson  \cite{MOS}, Arambasi\'c--Baki\'c--Moslehian \cite{ABM} have investigated the Cauchy--Schwarz inequality and its various reverses in the framework of $C^*$-algebras and Hilbert $C^*$-modules. An application of the covariance-variance inequality to the Cauchy--Schwarz inequality was obtained by Fujii--Izumino--Nakamoto--Seo \cite{nak}. Some operator versions of the Cauchy--Schwarz inequality with simple conditions for the case of equality are presented by Fujii \cite{FUJ}.

In 1965, Callebaut \cite{CAL} gave the following refinement of the Cauchy--Schwarz inequality:

{\it Given a real number $s$, non-proportional sequences of positive real numbers $\{a_i\}_{i=1}^n, \{b_i\}_{i=1}^n$, the function $f(r,s)=\left(\sum_{i=1}^na_i^{s+r}b_i^{s-r}\right)\left(\sum_{i=1}^na_i^{s-r}b_i^{s+r}\right)$ is increasing in $0\leq |r| \leq 1$. If $\{a_i\}_{i=1}^n, \{b_i\}_{i=1}^n$ are proportional, then this expression is independent of $r$.}

Thus one can obtain many well-ordered inequalities lying between the left and the right sides of the Cauchy--Schwarz inequality. In particular, if $0\leq t\leq
s\leq \frac{1}{2}$ or $\frac{1}{2}\leq s\leq t\leq 1$, then
\begin{eqnarray} \label{CAL}
\left( \sum_{j=1}^{m}a_{j}^{1/2}b_{j}^{1/2}\right) ^{2}\leq \left(
\sum_{j=1}^{m}a_{j}^{s}b_{j}^{1-s}\right) \left(
\sum_{j=1}^{m}a_{j}^{1-s}b_{j}^{s}\right) &\leq &\left(
\sum_{j=1}^{m}a_{j}^{t}b_{j}^{1-t}\right) \left(
\sum_{j=1}^{m}a_{j}^{1-t}b_{j}^{t}\right)  \notag \\
\text{ \ \ \ \ \ \ \ \ \ \ \ \ } &\leq &\left( \sum_{j=1}^{m}a_{j}\right)
\left( \sum_{j=1}^{m}b_{j}\right)
\end{eqnarray}
for all positive real numbers $a_{j},b_{j}\,\,(1\leq j\leq m)$. This triple inequality is well-known as the Callebaut inequality. Applying H\"older's inequality, McLaughlin and Metcalf \cite{MM} obtained it in a simple fashion. The method of Callebaut may be used for finding other interesting refinements of classical inequalities, see \cite{EM, AAN} and references therein. Wada \cite{WAD} gave an operator version of the Callebaut inequality by showing that if $A$ and $B$ are positive operators on a Hilbert space and if $\sigma$ is an operator mean, then $$(A\sharp B)\otimes(A\sharp B)\leq\frac 12\left\{(A\sigma B)\otimes(A\sigma^\bot B)+(A\sigma^\bot B)\otimes(A\sigma B)\right\}\leq \\ \frac 12\{(A\otimes B)+(B\otimes A)\}\,.$$

The purpose of the paper is to present some noncommutative versions of the Callebaut inequality. More precisely, we give an operator Callebaut inequality involving the interpolation paths and apply it to the weighted operator geometric means. We also establish a matrix version of the Callebaut inequality and as a consequence obtain an inequality including the Hadamard product of matrices.
\section{Callebaut inequality for Hilbert space operators}

Our first operator version of \eqref{CAL} reads as follows:

\begin{theorem}\label{thmain}
\label{thmcalle} Let $A_{j},B_{j}\in \mathcal{P}\,\,(1\leq j\leq m)$ and $\sigma$ be an operator mean. Then
\begin{align}
\label{1}
\sum_{j=1}^{m}\left(A_{j}\sharp B_{j}\right)\leq \left( \sum_{j=1}^{m}A_{j}\sigma B_{j}\right) \sharp
\left(\sum_{j=1}^{m} A_{j}\sigma^{\bot}B_{j}\right)\leq \left(\sum_{j=1}^{m}A_{j}\right)\sharp \left(\sum_{j=1}^{m}B_{j}\right)\,.
\end{align}
Furthermore, if $\sigma_t$ is an iterpolational path for $\sigma$ such that $\sigma_t^\bot=\sigma_{1-t}$, then
\begin{align}
\label{2} \left( \sum_{j=1}^{m}A_{j}\sigma_{s}B_{j}\right) \sharp
\left(\sum_{j=1}^{m} A_{j}\sigma_{1-s}B_{j}\right) \leq \left( \sum_{j=1}^{m}A_{j}\sigma_{t}B_{j}\right) \sharp
\left(\sum_{j=1}^{m} A_{j}\sigma_{1-t}B_{j}\right)
\end{align}
for $s$ between $t$ and $1-t$.
\end{theorem}
\begin{proof}
Let $f$ be the representing function of $\sigma$. Then
\begin{eqnarray}\label{sharp}
(A \sigma B)\sharp(A\sigma^{\bot}B)&=&\left(A^{1/2} f(A^{-1/2}BA^{-1/2}) A^{1/2}\right)\sharp\left(A^{1/2} (A^{-1/2}BA^{-1/2})f(A^{-1/2}BA^{-1/2})^{-1}A^{1/2}\right)\nonumber\\
&=&A^{1/2}\left( f(A^{-1/2}BA^{-1/2})\sharp \left((A^{-1/2}BA^{-1/2})f(A^{-1/2}BA^{-1/2})^{-1} \right) \right)A^{1/2}\nonumber\\
&=&A^{1/2}\left(A^{-1/2}BA^{-1/2}\right)^{1/2}A^{1/2}\nonumber\\
&=&A\sharp B
\end{eqnarray}
for all $A,B \in \mathcal{P}$. It follows from \eqref{add} that
\begin{eqnarray*}
\sum_{j=1}^{m}(A_{j}\sigma B_{j}) \leq (\sum_{j=1}^{m}A_{j})\sigma (\sum_{j=1}^{m} B_{j})\\
\sum_{j=1}^{m}(A_{j}\sigma^{\bot} B_{j}) \leq (\sum_{j=1}^{m}A_{j})\sigma^{\bot} (\sum_{j=1}^{m} B_{j})\,,
\end{eqnarray*}
whence
\begin{eqnarray*}
\left( \sum_{j=1}^{m}A_{j} \sigma B_{j}\right) \sharp
\left(\sum_{j=1}^{m} A_{j}\sigma^{\bot}B_{j}\right) &\leq& \left((\sum_{j=1}^{m}A_{j})\sigma (\sum_{j=1}^{m} B_{j})\right)\sharp\left((\sum_{j=1}^{m}A_{j})\sigma^{\bot}(\sum_{j=1}^{m} B_{j})\right)\\
&&\qquad\qquad\qquad\qquad\qquad\qquad\qquad\quad \qquad (\mbox{by} \eqref{sub})\\
&=&(\sum_{j=1}^{m}A_{j})\sharp (\sum_{j=1}^{m} B_{j})\,, \qquad\qquad\qquad\quad\qquad (\mbox{by} \eqref{sharp})
\end{eqnarray*}
which gives the second inequality of \eqref{1}. Now we prove the first inequality of \eqref{1}:
\begin{eqnarray*}
\left( \sum_{j=1}^{m}A_{j}\sigma B_{j}\right) \sharp
\left(\sum_{j=1}^{m} A_{j}\sigma^{\bot} B_{j}\right) &\geq& \sum_{j=1}^{m}\left((A_{j}\sigma B_j)\sharp (A_{j}\sigma^{\bot}B_{j})\right)\\
&=&\sum_{j=1}^{m}A_{j}\sharp B_{j}\,.  \qquad\qquad\qquad\qquad\qquad\quad (\mbox{by} \eqref{sharp})
\end{eqnarray*}
Next we prove \eqref{2}. Replacing $A_j$ and $B_j$ by $A_j\sigma_t
B_j$ and $A_j\sigma_{1-t}B_j$, respectively, in \eqref{1} and noting
to $\sigma_s^\bot=\sigma_{1-s}$ we get
\begin{align*}
\left( \sum_{j=1}^{m}(A_j\sigma_t B_j)\sigma_{s}(A_j\sigma_{1-t}B_j)\right) &\sharp
\left(\sum_{j=1}^{m} (A_j\sigma_t B_j)\sigma_{1-s}(A_j\sigma_{1-t}B_j)\right)\\
&\leq \sum_{j=1}^{m}(A_j\sigma_t B_j)\sharp \sum_{j=1}^{m}(A_j\sigma_{1-t}B_j)\,.
\end{align*}
The first term of the inequality above, by \eqref{sigma}, is
\begin{align*}
\left( \sum_{j=1}^{m}(A_j\sigma_t B_j)\sigma_{s}(A_j\sigma_{1-t}B_j)\right) &\sharp
\left(\sum_{j=1}^{m} (A_j\sigma_t B_j)\sigma_{1-s}(A_j\sigma_{1-t}B_j)\right)\\
&=\left( \sum_{j=1}^{m}(A_j\sigma_{ts+(1-t)(1-s)}B_j)\right)\sharp \left(\sum_{j=1}^{m}(A_j\sigma_{(1-t)s+t(1-s)}B_j)\right)\\
&=\left( \sum_{j=1}^{m}(A_j\sigma_{ts+(1-t)(1-s)}B_j)\right)\sharp \left(\sum_{j=1}^{m}(A_j\sigma_{1-(ts+(1-t)(1-s))}B_j)\right)\,.
\end{align*}
If $s$ is between $t$ and $1-t$, then there is $s_0\in [0,1]$ such
that $s = ts_0+(1-t)(1-s_0)$. This shows that the desired inequality
holds.
\end{proof}
It follows from \eqref{path} that $Am_{0,t}B=A\sharp_t B$. Due to $A\sharp_t^\bot B=A\sharp_{1-t}B$, we infer that
\begin{corollary}\label{thmain1}
\label{thmcalle} Let $A_{j},B_{j}\in \mathcal{P}$. Then
\begin{align*}
\sum_{j=1}^{m}A_{j}\sharp B_{j}\leq \left( \sum_{j=1}^{m}A_{j}\sharp _{s}B_{j}\right) \sharp
\left(\sum_{j=1}^{m} A_{j}\sharp _{1-s}B_{j}\right)\leq \sum_{j=1}^{m}A_{j}\sharp \sum_{j=1}^{m}B_{j}\,.
\end{align*}
Moreover,
\begin{align*}
\left( \sum_{j=1}^{m}A_{j}\sharp _{s}B_{j}\right) \sharp
\left(\sum_{j=1}^{m} A_{j}\sharp _{1-s}B_{j}\right) \leq \left(
\sum_{j=1}^{m}A_{j}\sharp _{t}B_{j}\right) \sharp
\left(\sum_{j=1}^{m} A_{j}\sharp _{1-t}B_{j}\right)
\end{align*}
for $s$ between $t$ and $1-t$.
\end{corollary}
Taking positive scalars $a_j$ and $b_j$ for $A_j$ and $B_j$,
respectively, in Theorem \ref{thmain} we get
\begin{corollary}
The classical Callebaut inequality \eqref{CAL} holds.
\end{corollary}
\section{Callebaut inequality for matrices}

To achieve some Callebaut inequalities for matrices, we need the following lemma.

\begin{lemma}
\label{le7} Let $0\leq r\leq 1$. Then $A^{r}+A^{-r}\leq A+A^{-1}$ for all $A\in \mathcal{P}_{n}$.
\end{lemma}
\begin{proof}
Suppose that $A=U\Gamma U^{\ast }$ with unitary $U$ and diagonal matrix $\Gamma $. Then
\begin{align*}
A^{r}+A^{-r}& =U(\Gamma ^{r}+\Gamma ^{-r})U^{\ast } \\
& \leq U(\Gamma +\Gamma ^{-1})U^{\ast }=A+A^{-1}
\end{align*}
since $t^{r}+t^{-r}\leq t+t^{-1}$ for any positive real number $t$ and $0\leq r\leq 1$.
\end{proof}
\begin{theorem}
\label{e:th} The function
\begin{equation*}
f(t)=A^{1+t}\otimes B^{1-t}+A^{1-t}\otimes B^{1+t}
\end{equation*}
is decreasing on the interval $[-1,0]$, increasing on the interval $[0,1]$ and attains its minimum at $t=0$ for all $A,B\in \mathcal{P}_{n}$.
\end{theorem}
\begin{proof}
Let $0\leq \alpha \leq \beta $. Taking $0\leq r=\frac{\alpha }{\beta }\leq 1$ and replacing $A$ by $A^{\beta}\otimes B^{-\beta }$ in Lemma \ref{le7}, we get
\begin{equation*}
A^{\alpha }\otimes B^{-\alpha }+A^{-\alpha }\otimes B^{\alpha }\leq A^{\beta
}\otimes B^{-\beta }+A^{-\beta }\otimes B^{\beta }.
\end{equation*}
This further implies that
\begin{align*}
\left( A^{1/2}\otimes B^{1/2}\right) (A^{\alpha }\otimes B^{-\alpha
}+A^{-\alpha }& \otimes B^{\alpha })\left( A^{1/2}\otimes B^{1/2}\right)  \\
& \leq \left( A^{1/2}\otimes B^{1/2}\right) \left( A^{\beta }\otimes
B^{-\beta }+A^{-\beta }\otimes B^{\beta }\right) \left( A^{1/2}\otimes
B^{1/2}\right)\,,
\end{align*}
whence
\begin{equation*}
A^{1+\alpha }\otimes B^{1-\alpha }+A^{1-\alpha }\otimes B^{1+\alpha }\leq
A^{1+\beta }\otimes B^{1-\beta }+A^{1-\beta }\otimes B^{1+\beta }
\end{equation*}
Note that $f(t)=f(-t)$, so $f$ is decreasing on the interval
$[-1,0]$. The last statement on minimum point is obvious from the
preceding ones.
\end{proof}
\begin{corollary}
\label{co2.2} The function
\begin{equation*}
g(t)=A^{t}\otimes B^{1-t}+A^{1-t}\otimes B^{t}
\end{equation*}
is decreasing on $[0,1/2]$, increasing on $[1/2,1]$ and attains its minimum
at $t=\displaystyle\frac{1}{2}$ for all $A,B\in \mathcal{P}_{n}$.
\end{corollary}
\begin{proof}
The proof follows by replacing $A,B$ by $A^{1/2},B^{1/2}$ in Theorem \ref{e:th},  respectively, and then replacing $\displaystyle\frac{1+t}{2}$ by $t$.
\end{proof}
The following theorem is our second version of the Callebaut
inequality \eqref{CAL}.

\begin{theorem}\label{main}
\label{thmcalle}Let $A_{j},B_{j}\in \mathcal{P}_{n}$ ,$~1\leq j\leq m$. Then
\begin{align*}
2\sum_{j=1}^{m}\left( A_{j}\sharp B_{j}\right) & \otimes
\sum_{j=1}^{m}\left( A_{j}\sharp B_{j}\right) \\
& \leq \sum_{j=1}^{m}\left( A_{j}\sharp _{s}B_{j}\right) \otimes
\sum_{j=1}^{m}\left( A_{j}\sharp _{1-s}B_{j}\right) +\sum_{j=1}^{m}\left(
A_{j}\sharp _{1-s}B_{j}\right) \otimes \sum_{j=1}^{m}\left( A_{j}\sharp
_{s}B_{j}\right) \\
& \leq \sum_{j=1}^{m}\left( A_{j}\sharp _{t}B_{j}\right) \otimes
\sum_{j=1}^{m}\left( A_{j}\sharp _{1-t}B_{j}\right) +\sum_{j=1}^{m}\left(
A_{j}\sharp _{1-t}B_{j}\right) \otimes \sum_{j=1}^{m}\left( A_{j}\sharp
_{t}B_{j}\right) \\
& \leq \sum_{j=1}^{m}A_{j}\otimes
\sum_{j=1}^{m}B_{j}+\sum_{j=1}^{m}B_{j}\otimes \sum_{j=1}^{m}A_{j}.
\end{align*}
for $0\leq t\leq s\leq \frac{1}{2}$ or $\frac{1}{2}\leq s\leq t\leq 1$.
\end{theorem}
\begin{proof}
In order to prove the above inequalities we first prove that the function
\begin{equation*}
f(t)=\sum_{j=1}^{m}\left( A_{j}\sharp _{t}B_{j}\right) \otimes
\sum_{j=1}^{m}\left( A_{j}\sharp _{1-t}B_{j}\right) +\sum_{j=1}^{m}\left(
A_{j}\sharp _{1-t}B_{j}\right) \otimes \sum_{j=1}^{m}\left( A_{j}\sharp
_{t}B_{j}\right)
\end{equation*}
is decreasing on $[0,1/2]$, increasing on $[1/2,1]$ and attains its minimum
at $t=\displaystyle\frac{1}{2}$. By Corollary \ref{co2.2} the function $t\rightarrow C_{i}^{t}\otimes C_{j}^{1-t}+C_{i}^{1-t}\otimes C_{j}^{t},$
where $C_{j}=A_{j}^{-1/2}B_{j}A_{j}^{-1/2},$ is decreasing on $[0,1/2]$,
increasing on $[1/2,1]$ and attains its minimum at $t=\displaystyle\frac{1}{
2}.$ Hence the function
\begin{align*}
t& \rightarrow \left( A_{i}^{1/2}\otimes A_{j}^{1/2}\right) \left(
C_{i}^{t}\otimes C_{j}^{1-t}+C_{i}^{1-t}\otimes C_{j}^{t}\right) \left(
A_{i}^{1/2}\otimes A_{j}^{1/2}\right) \\
&~~~~~~~~~~~~~~~~ =\left( A_{i}\sharp _{t}B_{i}\right) \otimes \left(
A_{j}\sharp _{1-t}B_{j}\right) +\left( A_{i}\sharp _{1-t}B_{i}\right)
\otimes \left( A_{j}\sharp _{t}B_{j}\right)
\end{align*}
is decreasing on $[0,1/2]$, increasing on $[1/2,1]$ and attains its minimum
at $t=\displaystyle\frac{1}{2}.$ Since
\begin{align*}
\sum_{j=1}^{m}\left( A_{j}\sharp _{t}B_{j}\right) & \otimes
\sum_{j=1}^{m}\left( A_{j}\sharp _{1-t}B_{j}\right) +\sum_{j=1}^{m}\left(
A_{j}\sharp _{1-t}B_{j}\right) \otimes \sum_{j=1}^{m}\left( A_{j}\sharp
_{t}B_{j}\right) \\
& =\sum_{i,\text{ }j=1}^{m}\Big( \left( A_{i}\sharp
_{t}B_{i}\right)\otimes \left(A_{j}\sharp
_{1-t}B_{j}\right)+\left(A_{i}\sharp
_{1-t}B_{i}\right)\otimes\left(A_{j}\sharp _{t}B_{j}\right)\Big)
\end{align*}
therefore the function $f(t)$ is decreasing on $[0,1/2]$, increasing on $[1/2,1]$ and attains its minimum at $t=\displaystyle\frac{1}{2}.$ On using
the fact that $f(1/2)\leq f(s)\leq f(t)\leq f(1),$ for $\frac{1}{2}\leq
s\leq t\leq 1,$ or $f(1/2)\leq f(s)\leq f(t)\leq f(0),$ for $\frac{1}{2}\geq
s\geq t\geq 0,$ we get the desired result.
\end{proof}
We have the following corollary as a consequence to Theorem \ref{thmcalle}.

\begin{corollary}
Let $A_{j},B_{j}\in \mathcal{P}_{n}$ ,$~1\leq j\leq m$. Then
\begin{align*}
\sum_{j=1}^{m}\left( A_{j}\sharp B_{j}\right) \circ \sum_{j=1}^{m}\left(
A_{j}\sharp B_{j}\right)& \leq \sum_{j=1}^{m}\left( A_{j}\sharp
_{s}B_{j}\right) \circ \sum_{j=1}^{m}\left( A_{j}\sharp _{1-s}B_{j}\right) \\
&\leq \sum_{j=1}^{m}\left( A_{j}\sharp _{t}B_{j}\right) \circ
\sum_{j=1}^{m}\left( A_{j}\sharp _{1-t}B_{j}\right) \\
& \leq \sum_{j=1}^{m}A_{j}\circ \sum_{j=1}^{m}B_{j}
\end{align*}
for $0\leq t\leq s\leq \frac{1}{2}$ or $\frac{1}{2}\leq s\leq t\leq 1$.
\end{corollary}
The next result is a consequence of Theorem \ref{main} with
$B_j=I_n\,\,(1\leq j\leq n)$.
\begin{corollary}
\label{etet} Let $A_{j}\in \mathcal{P}_{n}$ ,$~1\leq j\leq m$. Then
\begin{eqnarray}\label{revise}
\left( \frac{1}{m}\sum_{j=1}^{m}A_{j}^{1/2}\right) \circ \left( \frac{1}{m}
\sum_{j=1}^{m}A_{j}^{1/2}\right) \leq \left( \frac{1}{m}
\sum_{j=1}^{m}A_{j}^{t}\right) \circ \left( \frac{1}{m}
\sum_{j=1}^{m}A_{j}^{1-t}\right) \leq \frac{1}{m}\sum_{j=1}^{m}(A_{j}\circ
I_{n})
\end{eqnarray}
for all $0\leq t\leq 1.$
\end{corollary}


\textbf{Acknowledgment.} The authors would like to sincerely thank
Professor S. Wada and the anonymous referee for carefully reading
the article and useful comments.

\end{document}